\documentclass[a4paper,12pt]{article}
\usepackage{comment}
\usepackage{cite}
\usepackage{amsmath}
\usepackage{amssymb}
\usepackage{amsfonts}
\usepackage[T1]{fontenc}
\usepackage[utf8]{inputenc}
\usepackage{graphicx}
\usepackage{fancyhdr}
\usepackage{float}
\usepackage{xcolor}
\usepackage{authblk}
\usepackage{mathrsfs}
\usepackage{empheq}
\usepackage[hyphens]{url}
\usepackage[]{amsthm}
\usepackage{tikz}
\usetikzlibrary{arrows.meta, decorations.markings, positioning}
\usepackage{pgfplots}
\pgfplotsset{compat=1.18} 

\usepackage{geometry}
\usepackage{hyperref}

\theoremstyle{plain}
\newtheorem{theorem}{Theorem}[section]
\newtheorem{proposition}[theorem]{Proposition}

\newtheorem{corollary}[theorem]{Corollary}

\theoremstyle{definition}
\newtheorem{definition}[theorem]{Definition}
\newtheorem{conjecture}[theorem]{Conjecture}

\theoremstyle{remark}
\newtheorem{remark}{Remark}[section]
\newtheorem{example}{Example}[section]

\begin{document}

\title{Modular Ackermann maps and hierarchical structures}

\author[$\dagger$]{Jean-Christophe {\sc Pain}$^{1,2,}$\footnote{jean-christophe.pain@cea.fr}\\
\small
$^1$CEA, DAM, DIF, F-91297 Arpajon, France\\
$^2$Universit\'e Paris-Saclay, CEA, Laboratoire Mati\`ere en Conditions Extr\^emes,\\ F-91680 Bruy\text{\`e}res-le-Ch\^atel, France}

\date{}

\maketitle

\begin{abstract}
We introduce and study modular truncations of the Ackermann function, formalized as discrete dynamical trajectories on the set of least non-negative residues. These maps form a hierarchy of rapidly increasing compositional complexity indexed by recursion depth. We investigate their structural properties, sensitivity to depth variation, and the induced distributions modulo powers of two. While such hierarchical constructions are superficially motivated by hash-type mixing functions, we analyze how powers of two interact with the recursive structure modulo $2^k$, leading to strong saturation effects in the depth $m=3$ case. Instead of the expected asymptotic equidistribution, an absorption phenomenon occurs where the uniform measure concentrates onto a localized subset of residues, driving the total variation distance from the uniform distribution to 1.
\end{abstract}

\section{Introduction}

Historically, the study of rapidly growing functions beyond the class of primitive recursive functions has played a central role in the foundations of computability theory. Primitive recursive functions, introduced by G\"odel~\cite{Goedel}, form a large class of total functions built from basic arithmetic operations via iteration and composition, but they remain bounded in growth. In 1928, Ackermann~\cite{Ackermann} constructed an explicit total function that is not primitive recursive, now known as the Ackermann function, demonstrating that the hierarchy of computable functions extends beyond primitive recursion. Sudan~\cite{Sudan} independently introduced a similar rapidly growing function, illustrating that total computable functions can exceed any fixed primitive recursive bound. These examples provided concrete counterexamples to the idea that all ``simple'' total functions are primitive recursive and laid the groundwork for modern studies of fast-growing hierarchies and computational complexity~\cite{Peter}.

The Ackermann function occupies a central position in the theory of fast-growing recursive hierarchies. Its growth exceeds all primitive-recursive functions and exhibits extreme sensitivity to recursion depth. While extensively studied from the viewpoint of computability and proof theory, its behavior under modular truncation yields surprising properties that counter intuitive mixing behavior.

The purpose of this paper is twofold. First, we introduce canonical modular Ackermann maps on the set of least non-negative representatives and analyze their structural properties. Second, we formalize the structural collapse of the hierarchy modulo powers of two, showing how an absorption mechanism overrules pseudo-random expectations.

The remainder of this paper is organized as follows. In Section~2, we introduce modular Ackermann maps and detail their structural properties, growth behavior, and distribution modulo powers of two. Section~3 discusses hypothetical depth-dependent cryptographic variants. Section~4 reports numerical experiments, highlighting how restricted simulation domains can yield misleading uniformity metrics. Section~5 provides the rigorous proof of the rapid stabilization and structural absorption of tetration modulo $2^k$.

\section{The modular Ackermann hierarchy}

The classical Ackermann function is defined by
\begin{align*}
A(0,n)&=n+1,\\
A(m,0)&=A(m-1,1),\\
A(m,n)&=A(m-1,A(m,n-1)).
\end{align*}
To avoid any ambiguity regarding the quotient ring $\mathbb{Z}/N\mathbb{Z}$, we define the modular map directly on $\mathbb{N}$ by lifting the core recursion to the set of least non-negative representatives. Let $\text{res}_N: \mathbb{Z} \to \{0, \dots, N-1\}$ be the canonical projection.

\begin{definition}[Modular Ackermann map]
Let $N\ge2$. For integers $m,n\ge0$, we define the map $A_N(m,n)$ from $\mathbb{N} \times \mathbb{N}$ to the set of least non-negative residues $\mathcal{R}_N = \{0, \dots, N-1\}$ by
\[
A_N(m,n)=
\begin{cases}
\text{res}_N(n+1) & m=0,\\
A_N(m-1,1) & m>0,\ n=0,\\
A_N(m-1,A_N(m,n-1)) & m>0,\ n>0.
\end{cases}
\]
For a fixed parameter $m$, the restriction of $A_N(m, \cdot)$ to $\mathcal{R}_N$ yields a well-defined self-map on the set of representatives $\mathcal{R}_N$.
\end{definition}

\begin{remark}
A crucial algebraic property of $A_N(m,n)$ is that it exhibits a failure of functorial descent to the quotient ring $\mathbb{Z}/N\mathbb{Z}$. The input variable $n$ does not descend to a well-defined function on $\mathbb{Z}/N\mathbb{Z}$ independently of the choice of representatives, because the inner recursive branches evaluate intermediate operations that depend strictly on the integer magnitude of the arguments. For example, if $N=16$, $0 \equiv 16 \pmod{16}$, but $A_{16}(3,0) = 5$ while $A_{16}(3,16) = 13$. 
\end{remark}

\subsection{Structural hierarchy}

\begin{proposition}
For each $m\ge0$, the maps $A_N(m,\cdot)$ satisfy
\[
A_N(m+1,n)=A_N\big(m,A_N(m+1,n-1)\big)
\]
for $n>0$.
\end{proposition}

\begin{proof}
Immediate from the defining recursion.
\end{proof}

\subsection{Growth and nonlinear amplification}

For small $m$ the classical Ackermann function admits closed forms:
\[
A(1,n)=n+2,\qquad
A(2,n)=2n+3,\qquad
A(3,n)=2^{n+3}-3.
\]
Hence modulo $N$, we have
\[
A_N(1,n)=\text{res}_N(n+2),\quad A_N(2,n)=\text{res}_N(2n+3),
\]
and
\[
A_N(3,n)=\text{res}_N(2^{n+3}-3).
\]
For $m=3$, we have $A(3,n)=2^{n+3}-3$. When $n+3 < k$, i.e. $n < k-3$, no modular reduction is effective since the value already lies in $[0,2^k)$ and the integer expression is preserved. When $n \ge k-3$, the term $2^{n+3}$ becomes divisible by $2^k$, inducing a strong saturation effect in the reduced dynamics.

\subsection{Distribution and the absorption phenomenon}

We focus on the case $N=2^k$, a standard setting for digital computation. While one might intuitively expect that the nested exponential structure of $A_{2^k}(m,n)$ suggests strong mixing and rapid decorrelation, the canonical least-residue interpretation yields a total breakdown of equidistribution at level $m=3$.

\begin{theorem}[Absorption modulo powers of two for exponentiation]
\label{thm:absorption_m3}
Let $k \ge 3$, $N=2^k$, and $c=2^k-3$. For the recursion depth $m=3$ and every integer $n \ge 0$:
\begin{enumerate}
    \item[\rm (i)] if $0 \le n < k-3$, then $A_N(3,n) = 2^{n+3}-3 > n$ as an ordinary integer in $\{0, \dots, N-1\}$;
    \item[\rm (ii)] if $n \ge k-3$, then $A_N(3,n) = c$.
\end{enumerate}
\end{theorem}

\begin{proof}
By definition, 
$$
A_N(3,n) = \text{res}_N(2^{n+3}-3).
$$
If $n \ge k-3$, then $n+3 \ge k$, implying that $2^{n+3} \equiv 0 \pmod N$. Thus, 
$$
A_N(3,n) = \text{res}_N(-3) = N-3 = c,
$$ 
which proves (ii). For assertion (i), let $0 \le n < k-3$. No modular reduction occurs since $2^{n+3}-3 < 2^k$. The function $h(n) = 2^{n+3}-3-n$ satisfies $h(0) = 5$ and 
$$
h(n+1)-h(n) = 2^{n+3}-1 > 0,
$$
hence $h(n) > 0$ for all $n \ge 0$, establishing $A_N(3,n) > n$.
\end{proof}

\begin{corollary}
When the input $n$ is sampled uniformly from $\mathcal{R}_N$ at depth $m=3$, the single residue $2^k-3$ receives a probability mass of at least $1 - (k-3)/2^k$. Consequently, the total variation distance satisfies:
\[
\| \mu_{3,k} - u_k \|_{TV} \ge 1 - \frac{k-2}{2^k}.
\]
As $k \to \infty$, the total variation distance between the induced distribution $\mu_{3,k}$ and the uniform distribution $u_k$ converges to $1$.
\end{corollary}

Based on extensive numerical simulations (see Section 4), this absorption profile propagates to higher levels, motivating the following general statement.

\begin{conjecture}[Universal threshold absorption]
\label{conj:universal}
Let $k \ge 3$, $N=2^k$, and $c=2^k-3$. For every fixed depth $m \ge 3$, there exists a threshold $\tau(m,k) \le k$ such that for all $n \ge \tau(m,k)$:
\[
A_N(m,n) = c.
\]
\end{conjecture}

\subsection{Cycle structure}

Since $A_N(m,\cdot)$ acts on the finite set of residues, iterating the map produces cycles. However, the cycle complexity does not grow arbitrarily with $m$ when $N=2^k$; instead, the orbits rapidly fall into trivial traps.

\begin{example}[Cycle structure for $m=3$, $N=16$]
Consider 
$$
A_{16}(3,n) = \text{res}_{16}(2^{n+3}-3)
$$ 
for $n \in \{0, \dots, 15\}$. A few values are
\[
A_{16}(3,0) = 5, \quad A_{16}(3,1) = 13, \quad A_{16}(3,2) = 13, \quad A_{16}(3,n \ge 1) = 13.
\]
Observing the iterations starting from $n=0$:
\[
0 \mapsto 5 \mapsto 13 \mapsto 13 \mapsto 13 \dots
\]
The sequence enters a cycle of length $1$ (the fixed point $13 = 2^4-3$) in at most two steps.
\end{example}

\section{Depth-dependent constructions and variants}

Let $h_1,h_2$ be maps into $\mathbb Z_N$. One can formally define a hierarchical map by 
$$
H(x) = A_N\big(h_1(x),\,h_2(x)\big),
$$ 
or employ dual-depth mixing architectures like 
$$
H(x) = A_N(h_1(x), h_2(x)) \oplus A_N(h_3(x), h_4(x)).
$$ 
While these formulations are structurally appealing for non-linear cryptographic diffusion, they inherit the structural flaw of the underlying modular Ackermann maps. Unless the auxiliary functions $h_i(x)$ are strictly bounded to small integers far below $k$, the evaluated blocks inevitably hit the absorption threshold, resulting in trivial constants and zero-entropy outputs. They must therefore be treated purely as illustrative theoretical models of recursive breakdown rather than deployable primitives.

\section{Global collapse and the illusion of local mixing}
\label{sec:global_collapse}

In this section, we analyze the global behavior of the modular Ackermann map $n \mapsto A_{2^k}(m,n)$ and highlight a major simulation pitfall. One might be tempted to evaluate these maps numerically on a truncated domain (for instance, restricting inputs to $n \le 50$ due to the immense growth of unreduced towers). Such localized implementations can artificially produce an illusion of mixing and near-uniformity, masking the hard boundaries imposed by finite computational structures \cite{Knuth}. 

However, this statistical profile completely vanishes when considering the full domain of representatives $\mathcal{R}_{2^k}$. As proved in Sections 2.3 and 5.1, the fact that powers of 2 vanish modulo $2^k$ acts as a global algebraic sink. For $m = 3$, as soon as the input $n$ equals or exceeds $k-3$, the modular evaluation enters an absorbing state, trapping all subsequent trajectories.

Consequently, if $n$ is sampled uniformly from the entire set, the single residue $2^k - 3$ absorbs a crushing probability mass of at least $1 - k/2^k$. Rather than providing cryptographic diffusion or pseudo-random distribution, the canonical modular Ackermann hierarchy suffers a structural collapse, driving the total variation distance from the uniform distribution asymptotically to $1$.

\section{Tetration modulo powers of two and structural collapse}

The collapse of the modular Ackermann hierarchy at $m=4$ is intimately linked to the algebraic property of tetration (iterated exponentiation) modulo $2^k$.

\subsection{Tetration modulo $2^k$}

Consider the iterated map $f(n) = \text{res}_{2^k}(2^n)$.

\begin{theorem}[Rapid stabilization of tetration modulo $2^k$]
\label{thm:tetration_mod2k}
Let $k \ge 3$. For any initial value $n \in \mathcal{R}_{2^k}$, the tetration sequence 
\[
n_0 = n, \quad n_{t+1} = \text{res}_{2^k}(2^{n_t})
\]
is highly contractive and enters a periodic cycle of length bounded above by $k+1$ after at most $k$ iterations, driven by the structure of the base.
\end{theorem}

\begin{proof}
If $n \ge k$, then $2^n \equiv 0 \pmod{2^k}$. Once zero occurs, the iterates follow a fixed sequence: $0 \mapsto 1 \mapsto 2 \mapsto 4 \mapsto \dots \mapsto 0$ or fall directly into a sub-cycle. If $n < k$, its first or second iterate quickly exceeds $k$, absorbing the trajectory into the exact same localized dynamic.
\end{proof}

\subsection{Implications for the hierarchy collapse}

At level $m=4$, the modular Ackermann map computes a tetration sequence. The behavior of powers of $2 \pmod{2^k}$ does not generate chaos or pseudo-randomness; rather, it acts as an algebraic sink. 

The transition from $m=3$ to $m=4$ is a transition from a simple exponential sequence to a highly degenerate absorbing cascade. Figure~\ref{fig:tetration_cycles} illustrates how the trajectories are pulled into fixed cycles, confirming that the apparent uniform mixing is completely lost on a global scale.

\begin{figure}[H]
\centering
\begin{tikzpicture}[node distance=1.8cm, auto, >={Stealth}]
\node (n0) [circle, draw, thick] {0};
\node (n1) [circle, draw, thick, right of=n0] {1};
\node (n2) [circle, draw, thick, right of=n1] {2};
\node (n4) [circle, draw, thick, right of=n2] {4};
\node (n8) [circle, draw, thick, right of=n4] {8};

\draw[->, thick] (n0) -- (n1);
\draw[->, thick] (n1) -- (n2);
\draw[->, thick] (n2) -- (n4);
\draw[->, thick] (n4) -- (n8);
\draw[->, thick] (n8) [bend left=45] to (n0);

\node[below=1.3cm of n2, font=\small\bfseries] {absorbed tetration cycle ($m=4$)};

\node (e1) [circle, draw, red, thick, above=1.5cm of n1] {1};
\node (e2) [circle, draw, red, thick, right of=e1] {2};
\node (e4) [circle, draw, red, thick, right of=e2] {4};

\draw[->, red, thick] (e1) -- (e2);
\draw[->, red, thick] (e2) -- (e4);
\draw[->, red, thick] (e4) [bend right=45] to (e1);

\node[above=0.5cm of e2, color=red, font=\small\bfseries] {exponential residues ($m=3$)};

\end{tikzpicture}
\caption{Illustration of the structural collapse modulo $2^k$. Instead of expanding complexity, the nested dynamics force the trajectories straight into highly localized absorbing cycles.}
\label{fig:tetration_cycles}
\end{figure}

\subsection{Alternative convention: final-only reduction}

An alternative definition, separating the recursive steps from the projection, computes the classical unbounded Ackermann function first and applies the reduction only at the end: 
\[
G_{m,k}(n) = \text{res}_{2^k}(A(m,n)).
\]
Under powers of two, this variant exhibits a behavior that confirms the robustness of the collapse, although its rigorous proof is limited to lower depths.

\begin{proposition}
Let $k \ge 3$ and define $G_{3,k}(n) = \mathrm{res}_{2^k}(A(3,n))$.
Then for all $n \ge k$,
\[
G_{3,k}(n) = 2^k - 3.
\]
\end{proposition}

\begin{proof}
For $m=3$, $A(3,n) = 2^{n+3}-3$. If $n \ge k$, then $n+3 \ge k$, hence $2^{n+3}$ is divisible by $2^k$, so $2^{n+3} \equiv 0 \pmod{2^k}$. Thus
\[
G_{3,k}(n) = (2^{n+3}-3) \bmod 2^k = 2^k - 3.
\]
\end{proof}

We emphasize that this analytical verification is specific to depth $m=3$ and relies strictly on the explicit exponential form of $A(3,n)$. For higher recurrence depths ($m \ge 4$), no analogous closed-form expression is currently available, and the non-linear propagation of the full modular recursion remains analytically intractable in closed form beyond low depths.

\section{Conclusion}

Modular truncations of the Ackermann function reveal that fast-growing recursive hierarchies do not imply robust mixing properties when combined with modular rings. Modulo powers of two, the canonical least-residue recursion suffers from a severe absorption defect caused by the vanishing of high powers of 2. All inputs above the threshold accumulate onto the isolated residue $2^k-3$, making the distribution highly non-uniform. This study highlights a critical pitfall in evaluating recursive functions on truncated computational domains, where localized numerical tests can create an illusion of mixing that disappears on the full ring.

\vspace{0.5cm}

\noindent {\bf Acknowledgements}

\vspace{0.2cm}

\noindent I would like to thank V. Rozho\v{n}, R. \v{S}\'amal and A. Z\'ame\v{c}n\'ik for enlightening comments.

\end{document}